\documentclass[12pt, twoside, leqno]{article}

\usepackage{amsmath,amsthm}
\usepackage{amssymb}

\usepackage{enumitem}

\usepackage{graphicx}

\usepackage[T1]{fontenc}

\newtheorem{theorem}{Theorem}[section]

\newtheorem{lemma}[theorem]{Lemma}

\theoremstyle{definition}
\newtheorem{definition}[theorem]{Definition}
\newtheorem{remark}[theorem]{Remark}

\numberwithin{equation}{section}

\begin{document}


\baselineskip=17pt


\title{Density of sequences of the form $x_n=f(n)^n$ in $[0,1]$}

\author{J.C. Saunders\\
Department of Mathematics and Statistics\\
Mathematical Sciences 476\\
University of Calgary\\
2500 University Drive NW\\
T2N 1N4 Calgary, Alberta, Canada\\
E-mail: john.saunders1@ucalgary.ca}

\date{}

\maketitle


\renewcommand{\thefootnote}{}



\renewcommand{\thefootnote}{\arabic{footnote}}
\setcounter{footnote}{0}


\begin{abstract}
In 2013, Strauch asked how various sequences of real numbers defined from trigonometric functions such as $x_n=(\cos n)^n$ distributed themselves$\pmod 1$. Strauch's inquiry is motivated by several such distribution results. For instance, Luca proved that the sequence $x_n=(\cos \alpha n)^n\pmod 1$ is dense in $[0,1]$ for any fixed real number $\alpha$ such that $\alpha/\pi$ is irrational. Here we generalise Luca's results to other sequences of the form $x_n=f(n)^n\pmod 1$. We also examine the size of the set $|\{n\leq N:r<|\cos(n\pi\alpha)|^n\}|$, where $0<r<1$ and $\alpha$ are fixed such that $\alpha/\pi$ is irrational.    
\end{abstract}

\section{Introduction}
In 2013, Strauch \cite{strauch} asked how various sequences of real numbers defined from trigonometric functions, such as $x_n=(\cos n)^n$, distributed themselves$\mod 1$. Strauch's inquiry is motivated by several such distribution results. In 1995, Berend, Boshernitzan, and Kolesnik \cite{berend1} proved that the sequence $x_n=(\log n)\cos(n\alpha)\pmod 1$ is dense in $[0,1]$. In 1997 Bukor \cite{bukor} proved that the sequence $x_n=(\cos n)^n$ is dense in $[-1,1]$, which was generalised to $x_n=(\cos\alpha n)^n$ for any fixed $\alpha\in\mathbb{R}$ such that $\alpha/\pi$ is irrational by Luca \cite{luca}. There have also been other results on other properties of the distribution of such sequences. First, a definition.
\begin{definition}
A \textit{distribution function} $g(x)$ of a sequence $(x_n)_n$ such that $0\leq x_n\leq 1$ is a non-decreasing function with $g(0)=0$, $g(1)=1$, and such that there exists an increasing sequence of natural numbers $(N_k)_k$ such that for every point of continuity $x$ of $g$, we have
\begin{equation*}
g(x)=\lim_{k\rightarrow\infty}\frac{|\left\{n\leq N_k:x_n\leq x\right\}|}{N_k}.
\end{equation*}
The sequence is called \textit{uniformly distributed} if $g(x)=x$ is the only distribution function of the sequence.
\end{definition}
In 1953, Kuipers \cite{kuipers} proved that the sequence $x_n=\cos(n+\log n)$ is not uniformly distributed. More recently, in 2011, Berend and Kolesnik \cite{berend2} proved that the sequence $x_n=P(n)\cos n\alpha$ is uniformly distributed for any non-constant polynomial $P(n)$ and any $\alpha\in\mathbb{R}$ such that $\cos(\alpha)$ is transcendental. Also, in 2013, Aistleitner, Hofer, and Madritsch \cite{aistleitner} characterised all of the distribution functions of the sequence $x_n=(\cos n)^n$.
\newline
\newline
Here we generalise Luca's results to other sequences of the form $x_n=f(n)^n$. Our main results are the following:
\begin{theorem}\label{thm1}
Let $g$ be an analytic real-valued function on $[-1,1]$ and suppose that $|g(x)|\leq 1$ for all $-1\leq x\leq 1$ and $g(1)=1$. Also suppose $g(r)\neq 1$ for some $-1<r<1$. Also pick $\alpha\in\mathbb{R}$ such that $\frac{\alpha}{\pi}$ is irrational. Then the set $\{|g(\cos\alpha n)|^n:n\in\mathbb{N}\}$ is dense in $[0,1]$. Also, the set $\{|g(\sin\alpha n)|^n:n\in\mathbb{N}\}$ is dense in $[0,1]$.
\end{theorem}
\begin{theorem}\label{thm2}
Let $g$ be an analytic real-valued function on $[-1,1]$ and suppose that $|g(x)|\leq 1$ for all $-1\leq x\leq 1$ and $g(1)=1$. Also, suppose there exists $-1<r<1$ such that $g(r)\neq 1$. Also let $\alpha\in\mathbb{R}$. Then the set $\{|g(\cos\alpha n^2)|^n:n\in\mathbb{N}\}$ is dense in $[0,1]$.
\end{theorem}
\begin{theorem}\label{thm3}
Let $\alpha\in\mathbb{R}$. The set $\{\{\alpha n\}^n:n\in\mathbb{N}\}$ is dense in $[0,1]$ if and only if in the simple continued fraction expansion of $\alpha=[a_0;a_1,a_2,\ldots]$, we have
\begin{equation}
\limsup_{i\rightarrow\infty}a_{2i}=\infty.\label{eqn5}
\end{equation}
\end{theorem}
\begin{remark}
We deduce in Lemma \ref{lem5} that for any $\alpha\in\mathbb{R}$ the continued fraction expansion of $\alpha$ satisfies \eqref{eqn5} if and only if for every $\epsilon>0$ there exists $a,b\in\mathbb{Z}$ such that
\begin{equation*}
\frac{a}{b}-\frac{\epsilon}{b^2}\leq\alpha<\frac{a}{b}.    
\end{equation*}
In other words $\alpha$ is not badly approximable by rational numbers greater than $\alpha$. Note that this does not imply that $\alpha$ is not badly approximable by rational numbers less than $\alpha$ and that it is quite possible to construct an $\alpha$ as a counterexample. Indeed, it can be deduced that $\alpha$ will be a counter example if and only if the continued fraction expansion of $\alpha$ satisfies \eqref{eqn5}, but that there exists $C>0$ such that $a_{2i+1}<C$ for all $i\in\mathbb{N}$.
\end{remark}
Thanks to Luca \cite{luca} we know that the sequence $x_n=|\cos \alpha n|^n$ is dense in $[0,1]$, but we can also see that most of the terms in this sequence will be very close to $0$. This result begs the question on the density of the number of $N\in\mathbb{N}$ that satisfy $|\{n\leq N:r<|\cos(n\pi\alpha)|^n\}|$ for any fixed $0<r<1$. While it is to be expected that the density of such a set should be $0$, it turns out that we can construct an irrational number $\alpha$ that makes such a subset of $\mathbb{N}$ as large as desired so long as the density of this subset in $\mathbb{N}$ is still $0$.
\begin{theorem}\label{thm4}
Let $0<r<1$ and let $f(n):\mathbb{N}\rightarrow\mathbb{R}$ such that
\begin{equation}
\lim_{n\rightarrow\infty}f(n)=0.\label{eqn7}
\end{equation}
There exists $\alpha\in\mathbb{R}\setminus\mathbb{Q}$ such that there exist infinitely many $N\in\mathbb{N}$ such that
\begin{equation*}
|\{n\leq N:r<|\cos(n\pi\alpha)|^n\}|\geq f(N)\cdot N.
\end{equation*}
\end{theorem}
On the other hand, our final result gives us a bound that holds for any irrational $\alpha$.
\begin{theorem}\label{thm5}
Let $0<r<1$ and $\alpha\in\mathbb{R}\setminus\mathbb{Q}$. Then there exists infinitely many $N\in\mathbb{N}$ such that
\begin{equation*}
|\{n\leq N:r<|\cos(n\pi\alpha)|^n\}|\geq\frac{5^{1/4}(1-r)^{1/4}N^{1/4}}{2\sqrt{\pi}}.
\end{equation*}
\end{theorem}
\section{Proofs}
We require a couple of lemmas to prove Theorem \ref{thm1}.
\begin{lemma}[Luca\cite{luca}]\label{lem1}
Let $\zeta,a,b$, and $\gamma$ be real numbers with $\zeta$ irrational, $a,b>0$, $0<\gamma<1$. Also let $r=0,1,2$, or $3$ and $s=0$ or $1$. Then there exists infinitely many pairs of integers $m,n$ such that
\begin{equation*}
\zeta-\frac{m}{n}=\epsilon\left(\frac{a}{n^{1+\gamma}}+\Lambda\right),
\end{equation*}
where $\epsilon=\pm 1$ and $0<\Lambda<\frac{b}{n^{1+\gamma}\log n}$. Moreover, we can have the restrictions that $n>0$, $n\equiv r\pmod 4$, and $m\equiv s\pmod 2$.
\end{lemma}
\begin{lemma}\label{lem2}
Let $f:\mathbb{R}\rightarrow[0,1]$ be periodic with irrational period $T$. Also suppose $f(0)=1$ and that there exists $\delta>1$ such that
\begin{equation*}
\lim_{x\rightarrow 0}\frac{f(x)-1}{|x|^{\delta}}=q<0.
\end{equation*}
Then the set $\{f(n)^n:n\in\mathbb{N}\}$ is dense in $[0,1]$.
\end{lemma}
\begin{proof}
Let $c>0$ be a constant. By Lemma \ref{lem1} there exists infinitely many pairs of integers $m,n$ such that
\begin{equation*}
\left|\frac{1}{T}-\frac{m}{n}\right|=\frac{c}{n^{1+1/\delta}}+\Lambda,
\end{equation*}
where $0<\Lambda<\frac{1}{n^{1+1/\delta}\log n}$. Then
\begin{equation*}
n-mT=\epsilon\left(\frac{Tc}{n^{1/\delta}}+\Lambda_1\right),
\end{equation*}
where $0<\Lambda_1<\frac{T}{n^{1/\delta}\log n}$. Thus there are infinitely many integers $n$ satisfying
\begin{equation*}
f(n)=f\left(\epsilon\left(\frac{Tc}{n^{1/\delta}}+\Lambda_1\right)\right),
\end{equation*}
where $0<\Lambda_1<\frac{T}{n^{1/\delta}\log n}$. Let these integers be $n_1,n_2,\ldots$. Then we have
\begin{align*}
\lim_{k\rightarrow\infty}\frac{f(n_k)-1}{\left(\frac{Tc}{n_k^{1/\delta}}+\Lambda_{1,k}\right)^{\delta}}=\lim_{k\rightarrow\infty}\frac{f\left(\epsilon_k\left(\frac{Tc}{n_k^{1/\delta}}+\Lambda_{1,k}\right)\right)-1}{\left(\frac{Tc}{n_k^{1/\delta}}+\Lambda_{1,k}\right)^{\delta}}=q.
\end{align*}
Thus we can deduce
\begin{equation*}
\lim_{k\rightarrow\infty}n_kf(n_k)-n_k=q(Tc)^{\delta}.
\end{equation*}
Hence $f(n_k)=1+\frac{q(Tc)^{\delta}}{n_k}(1+o(1))$ as $k\rightarrow\infty$. It follows that
\begin{equation*}
\lim_{k\rightarrow\infty}f(n_k)^{n_k}=e^{q(Tc)^{\delta}}.
\end{equation*}
Since $c>0$ was arbitrary, we have the result.
\end{proof}
\begin{proof}[Proof of Theorem \ref{thm1}]
Let $f(x)=|g(\cos \alpha x)|$. Then $f$ has range $[0,1]$, is periodic with period $\frac{2\pi}{\alpha}$, and $f(0)=1$. Since $g$ is analytic, we have that $h(x)=g(\cos\alpha x)$ is analytic on the real line and so $h(x)$ must have a Taylor series about $x=0$. Since $g$ isn't identically $1$, we can therefore deduce that there exists a positive integer $d$ such that the limit
\begin{equation*}
\lim_{x\rightarrow 0}\frac{f(x)-1}{x^d}.
\end{equation*}
exists, is finite, and is nonzero. Since
\begin{equation*}
0=h'(0)=\lim_{x\rightarrow 0}\frac{h(x)-1}{x}=\lim_{x\rightarrow 0}\frac{f(x)-1}{x}
\end{equation*}
$d$ must be at least $2$. Applying Lemma \ref{lem2} gives that the set $\{|g(\cos \alpha n)|^n:n\in\mathbb{N}\}$ is dense in $[0,1]$. Moreover, we can repeat this argument to derive that the set $\{|g(\cos \alpha n)|^n:n\in\mathbb{N},\text{ }n\equiv 1\pmod 4\}$ is also dense in $[0,1]$. Noticing that $\sin(n\alpha)=\cos\left(n\left(\frac{\pi}{2}-\alpha\right)\right)$ also gives that the set $\{|g(\sin\alpha n)|^n:n\in\mathbb{N}\}$ is dense in $[0,1]$.
\end{proof}
The proof of Theorem \ref{thm2} follows the same proof technique as Theorem \ref{thm1} only we require analogous results to Lemmas \ref{lem1} and \ref{lem2} with $n$ being replaced by $n^2$. Such analogous results can be proved from the following theorem of Zaharescu \cite{zaharescu}.
\begin{theorem}[Zaharescu\cite{zaharescu}]\label{zaharescu}
Let $0<\theta<\frac{2}{3}$. Then for every real number $\zeta$ there exist infinitely many pairs of integers $m,n$ such that
\begin{equation}
\left|\zeta-\frac{m}{n^2}\right|<\frac{1}{n^{2+\theta}}.\label{eqn1}
\end{equation}
\end{theorem}
\begin{lemma}\label{lem3}
Let $\zeta,a,b$, and $\gamma$ be real numbers with $a,b>0$, $0<\gamma<\frac{2}{3}$. Then there exist infinitely many pairs of integers $m,n$ such that
\begin{equation*}
\zeta-\frac{m}{n^2}=\epsilon\left(\frac{a}{n^{2+\gamma}}+\Lambda\right),
\end{equation*}
where $\epsilon=\pm 1$ and $0<\Lambda<\frac{b}{n^{2+\gamma}\log n}$.
\end{lemma}
\begin{proof}
Without loss of generality, we may assume that $\zeta>0$. Let $\gamma<\theta<\frac{2}{3}$. By Theorem \ref{zaharescu}, there exist two increasing sequences of positive integers $(u_t)_t$ and $(v_t)_t$ such that \eqref{eqn1} holds with $m=u_t$ and $n=v_t$ for all $t\in\mathbb{N}$. For the rest of the proof we'll assume that $\zeta-\frac{u_t}{v_t^2}>0$. The negative case follows similarly. By \eqref{eqn1} we get
\begin{equation*}
v_t^2\zeta-u_t<v_t^{-\theta}.
\end{equation*}
Thus
\begin{equation*}
\frac{a}{(\zeta v_t^2-u_t)v_t^{\gamma}}>av_t^{\theta-\gamma}.
\end{equation*}
Let $A=\frac{a}{(\zeta v_t^2-u_t)v_t^{\gamma}}$. Let $N=\lfloor A^{1/(2+\gamma)}+1\rfloor$. Since $N$ increases with $t$, we can choose $A$ sufficiently large such that $N<2A$. Then $A<N^{2+\gamma}$ so that we can deduce
\begin{equation*}
\frac{a}{N^{2+\gamma}v_t^{2+\gamma}}<\zeta-\frac{u_t}{v_t^2}.
\end{equation*}
To complete the proof it suffices to show that
\begin{equation*}
\zeta-\frac{u_t}{v_t^2}-\frac{a}{N^{2+\gamma}v_t^{2+\gamma}}<\frac{b}{N^{2+\gamma}v_t^{2+\gamma}\log(Nv_t)}
\end{equation*}
or
\begin{equation*}
N^{2+\gamma}<\frac{a}{(\zeta v_t^2-u_t)v_t^{\gamma}}+\frac{b}{(\zeta v_t^2-u_t)v_t^{\gamma}\log(Nv_t)}=A+\frac{Ab}{a(\log N+\log v_t)}.
\end{equation*}
It therefore suffices to show that
\begin{equation*}
(A^{1/(2+\gamma)}+1)^{2+\gamma}<A+\frac{Ab}{a(\log(2A)+\log v_t)}.
\end{equation*}
or
\begin{equation}
\left(1+\frac{1}{A^{1/(2+\gamma)}}\right)^{2+\gamma}-1<\frac{b}{a(\log(2A)+\log v_t)}.\label{eqn2}
\end{equation}
Since $A>av_t^{\theta-\gamma}$, we have
\begin{equation*}
v_t<\left(\frac{A}{a}\right)^{1/(\theta-\gamma)}
\end{equation*}
so that \eqref{eqn2} follows if
\begin{equation*}
\left(1+\frac{1}{A^{1/(2+\gamma)}}\right)^{2+\gamma}-1<\frac{b}{a(\log(2A)+(\theta-\gamma)^{-1}\log(A/a))}.
\end{equation*}
Noticing that
\begin{equation*}
\left(1+\frac{1}{A^{1/(2+\gamma)}}\right)^{2+\gamma}-1=O\left(\frac{1}{A^{1/(2+\gamma)}}\right)
\end{equation*}
and
\begin{equation*}
\frac{b}{a(\log(2A)+(\theta-\gamma)^{-1}\log(A/a))}=O\left(\frac{1}{\log A}\right)
\end{equation*}
we can choose $t$ and therefore $A$ large enough so that we have our result.
\end{proof}
\begin{lemma}
Let $f:\mathbb{R}\rightarrow[0,1]$ be periodic and $\alpha\in\mathbb{R}$. Also suppose $f(0)=1$ and that there exists $\delta>1$ such that
\begin{equation*}
\lim_{x\rightarrow 0}\frac{f(x)-1}{|x|^{\delta}}=q<0.
\end{equation*}
Then the set $\{f(\alpha n^2)^n:n\in\mathbb{N}\}$ is dense in $[0,1]$.
\end{lemma}
\begin{proof}
Let $c>0$ be a constant and let $T$ be the period of $f$. By Lemma \ref{lem3} there exists infinitely many pairs of integers $m,n$ such that
\begin{equation*}
\left|\frac{\alpha}{T}-\frac{m}{n^2}\right|=\frac{c}{n^{2+1/\delta}}+\Lambda,
\end{equation*}
where $0<\Lambda<\frac{1}{n^{2+1/\delta}\log n}$. Then
\begin{equation*}
\alpha n^2-mT=\epsilon\left(\frac{Tc}{n^{1/\delta}}+\Lambda_1\right),
\end{equation*}
where $0<\Lambda_1<\frac{T}{n^{1/\delta}\log n}$. Thus there are infinitely many integers $n$ satisfying
\begin{equation*}
f(\alpha n^2)=f\left(\epsilon\left(\frac{Tc}{n^{1/\delta}}+\Lambda_1\right)\right),
\end{equation*}
where $0<\Lambda_1<\frac{T}{n^{1/\delta}\log n}$. Let these integers be $n_1,n_2,\ldots$. Then we have
\begin{align*}
\lim_{k\rightarrow\infty}\frac{f(\alpha n_k^2)-1}{\left(\frac{Tc}{n_k^{1/\delta}}+\Lambda_{1,k}\right)^{\delta}}=\lim_{k\rightarrow\infty}\frac{f\left(\epsilon_k\left(\frac{Tc}{n_k^{1/\delta}}+\Lambda_{1,k}\right)\right)-1}{\left(\frac{Tc}{n_k^{1/\delta}}+\Lambda_{1,k}\right)^{\delta}}=q.
\end{align*}
Thus we can deduce
\begin{equation*}
\lim_{k\rightarrow\infty}n_kf(\alpha n_k^2)-n_k=q(Tc)^{\delta}.
\end{equation*}
Hence $f(\alpha n_k^2)=1+\frac{q(Tc)^{\delta}}{n_k}(1+o(1))$ as $k\rightarrow\infty$. It follows that
\begin{equation*}
\lim_{k\rightarrow\infty}f(\alpha n_k^2)^{n_k}=e^{q(Tc)^{\delta}}.
\end{equation*}
Since $c>0$ was arbitrary, we have the result.
\end{proof}
\begin{proof}[Proof of Theorem \ref{thm2}]
Let $f(x)=|g(\cos x)|$. Then $f$ has range $[0,1]$, is periodic, and $f(0)=1$. Since $g$ is analytic, we have that $h(x)=g(\cos x)$ is analytic on the real line and so $h(x)$ must have a Taylor series about $x=0$. Since $g$ isn't identically $1$, we can therefore deduce that there exists a positive integer $d$ such that the limit
\begin{equation*}
\lim_{x\rightarrow 0}\frac{f(x)-1}{x^d}.
\end{equation*}
exists, is finite, and is nonzero. Since
\begin{equation*}
0=h'(0)=\lim_{x\rightarrow 0}\frac{h(x)-1}{x}=\lim_{x\rightarrow 0}\frac{f(x)-1}{x}
\end{equation*}
$d$ must be at least $2$. Applying Lemma \ref{lem3} gives that the set $\{|g(\cos \alpha n^2)|^n:n\in\mathbb{N}\}$ is dense in $[0,1]$.
\end{proof}
\begin{lemma}\label{lem5}
Let $\alpha\in\mathbb{R}$ have simple continued fraction expansion $\alpha=[a_0;a_1,a_2,\ldots]$. The following two statements are equivalent:
\newline
1) There exists $c>0$ such that for all $a,b\in\mathbb{Z}$, $b\geq 0$ with $\alpha<\frac{a}{b}$, we have $\frac{c}{b^2}<\frac{a}{b}-\alpha$.
\newline
2) There exists $M>0$ such that $a_i<M$ for all even values of $i$.
\end{lemma}
\begin{proof}
We first show that 1) implies 2). Suppose there exists $c>0$ such that for all $a,b\in\mathbb{Z}$, $b\neq 0$ with $\alpha<\frac{a}{b}$ we have $\frac{c}{b^2}<\frac{a}{b}-\alpha$. Let $\frac{p_n}{q_n}$ be the $n$th convergent in the continued fraction expansion of $\alpha$. Suppose that $n$ is odd. Then we have that $\alpha<\frac{p_n}{q_n}$. Thus
\begin{equation*}
\frac{p_n}{q_n}-\alpha<\frac{1}{q_nq_{n+1}}.
\end{equation*}
Since $q_{n+1}=a_{n+1}q_n+q_{n-1}$, we have $q_{n+1}>a_{n+1}q_n$. Thus
\begin{equation*}
\frac{c}{q_n^2}<\frac{p_n}{q_n}-\alpha<\frac{1}{a_{n+1}q_n^2}.
\end{equation*}
So $a_{n+1}<c^{-1}$ and 2) follows.
\newline
\newline
We now show that 2) implies 1). Suppose there exists $M>0$ such that $a_i<M$ for all even values of $i$. Let $\alpha<\frac{a}{b}$. If $\frac{a}{b}$ isn't a convergent of $\alpha$, then $\frac{1}{2b^2}<\frac{a}{b}-\alpha$. Now suppose $\frac{a}{b}=\frac{p_n}{q_n}$, the $n$th convergent. Since $\alpha<\frac{p_n}{q_n}$, $n$ is odd. We have
\begin{equation*}
\frac{1}{2q_nq_{n+1}}<\frac{p_n}{q_n}-\alpha.
\end{equation*}
Since $q_{n+1}=a_{n+1}q_n+q_{n-1}$, we have $q_{n+1}<(a_{n+1}+1)q_n$. So
\begin{equation*}
\frac{1}{2(M+1)q_n^2}<\frac{1}{2(a_{n+1}+1)q_n^2}<\frac{p_n}{q_n}-\alpha.
\end{equation*}
Letting $c=\min\{\frac{1}{2},\frac{1}{2(M+1)}\}$, we have 1).
\end{proof}
To prove Theorem \ref{thm3} we first prove the following lemma.
\begin{lemma}\label{lem6}
Let $\alpha\in\mathbb{R}$. Suppose for all $\delta>0$ there exists $m,n\in\mathbb{Z}$ such that $\alpha<\frac{m}{n}$ and $\frac{m}{n}-\frac{\delta}{n^2}<\alpha$. Let $a,\epsilon>0$. Then there exists $m,n\in\mathbb{Z}$ such that
\begin{equation*}
a<n^2\left(\frac{m}{n}-\alpha\right)<a+\epsilon.
\end{equation*}
\end{lemma}
\begin{proof}
By our hypothesis, there exists two sequences $(u_t)_t,(v_t)_t$ in $\mathbb{N}$ and a function $f:\mathbb{N}\rightarrow\mathbb{R^+}$ with $\lim_{n\rightarrow\infty}f(n)=0$ such that
\begin{equation}
\alpha=\frac{u_t}{v_t}-\frac{f(v_t)}{v_t^2}\label{eqn3}
\end{equation}
for all $t\in\mathbb{N}$. Let $A=\frac{a}{f(v_t)}$ and let $N=\lfloor A^{1/2}+1\rfloor$. Since $\lim_{t\rightarrow\infty}f(v_t)=0$, we can choose $t$ sufficiently large so that
\begin{equation}
a<\frac{N^2a}{A}<a+\epsilon.\label{eqn4}
\end{equation}
By \eqref{eqn3}, we can derive that $A=\frac{a}{v_t(u_t-\alpha v_t)}$. Thus
\begin{equation*}
\frac{N^2a}{A}=N^2v_t^2\left(\frac{Nu_t}{Nv_t}-\alpha\right)
\end{equation*}
so that by \eqref{eqn4} we have our result.
\end{proof}
\begin{proof}[Proof of Theorem \ref{thm3}]
Let  $\alpha=[a_0;a_1,a_2,\ldots]$. First suppose that \eqref{eqn5} doesn't hold. Then by Lemma \ref{lem5}, we have there exists $c>0$ such that for all $m,n\in\mathbb{Z}$ with $n\alpha<m$ we have $n\alpha<m-\frac{c}{n}$. It follows that $\{n\alpha\}<1-\frac{c}{n}$ for all $n>c$. Thus
\begin{equation*}
\limsup_{n\rightarrow\infty}\{n\alpha\}^n\leq e^{-c}<1.
\end{equation*}
Thus $\{\{\alpha n\}^n:n\in\mathbb{N}\}$ isn't dense in $[0,1]$.
\newline
\newline
Conversely, suppose \eqref{eqn5} holds. Then by Lemma \ref{lem5}, we have that for all $\delta>0$ there exists $m,n\in\mathbb{Z}$ such that $\alpha<\frac{m}{n}$ and $\frac{m}{n}-\frac{\delta}{n^2}<\alpha$. Let $0<a<1$. Lemma \ref{lem6} implies that there exists subsequences $(u_t)_t$ and $(v_t)_t$ in $\mathbb{Z}$ and $\mathbb{N}$ respectively such that
\begin{equation*}
\lim_{t\rightarrow\infty}v_t^2\left(\frac{u_t}{v_t}-\alpha\right)=-\log a.
\end{equation*}
Thus we can derive
\begin{equation*}
\{v_t\alpha\}=1+\frac{\log a}{v_t}(1+o(1))
\end{equation*}
for sufficiently large $t$. Hence
\begin{equation*}
\lim_{t\rightarrow\infty}\{v_t\alpha\}^{v_t}=a.
\end{equation*}
Since $0<a<1$ was arbitrary, $\{\{\alpha n\}^n:n\in\mathbb{N}\}$ is dense in $[0,1]$.
\end{proof}
\begin{proof}[Proof of Theorem \ref{thm4}]
Let $0<r<1$ and let $f(n):\mathbb{N}\rightarrow\mathbb{R}$ such that \eqref{eqn7} holds. Let $r<r'<1$. Define two increasing sequences of natural numbers $(d_i)_i$ and $(N_i)_i$ in the following way. First, pick $N_1\in\mathbb{N}$ such that $f(N_1)<1$. Then pick $d_1$ such that
\begin{equation*}
\frac{\log\left(\frac{2\pi N_1^{3/2}}{\sqrt{-2\log r'}}\right)}{\log(10)}<d_1.
\end{equation*}
We define the rest of the two sequences recursively in the following way. Suppose we have decided on the values of $N_i$ and $d_i$ for some $i\in\mathbb{N}$. Pick $N_{i+1}>N_i$ such that $f(N_{i+1})\leq\frac{1}{10^{d_i}}$. Then pick $d_{i+1}$ such that $2d_i\leq d_{i+1}$ and
\begin{equation}
\frac{\log\left(\frac{2\pi N_{i+1}^{3/2}}{\sqrt{-2\log r'}}\right)}{\log(10)}\leq d_{i+1}.\label{eqn6}
\end{equation}
Consider the number
\begin{equation*}
\alpha=\sum_{i=1}^{\infty}\frac{1}{10^{d_i}}.
\end{equation*}
Since $2d_i\leq d_{i+1}$ for all $i\in\mathbb{N}$ $\alpha$ is irrational. Let $k,m\in\mathbb{N}$, where $k\geq 2$, $m\leq\frac{N_k}{10^{d_{k-1}}}$ and $-2\log r'<N_k$. By \eqref{eqn6}, we have the following.
\begin{equation*}
\pi\{10^{d_{k-1}}m\alpha\}<\frac{2\pi m}{10^{d_k-d_{k-1}}}\leq\frac{2\pi N_k}{10^{d_k}}\leq\frac{\sqrt{-2\log r'}}{N_k^{1/2}}.
\end{equation*}
By the limit
\begin{equation*}
\lim_{x\rightarrow 0}\frac{1-\cos x}{x^2}=\frac{1}{2},
\end{equation*}
we have
\begin{equation*}
\limsup_{k\rightarrow\infty}\max_{\substack{1\leq m\leq\frac{N_k}{10^{d_{k-1}}}\\m\in\mathbb{N}}}\frac{(1-|\cos(10^{d_{k-1}}m\pi\alpha)|)N_k}{-2\log r'}\leq\frac{1}{2}.
\end{equation*}
We can therefore derive that
\begin{equation*}
1+\frac{\log r'}{N_k}(1+o_k(1))\leq\min_{\substack{1\leq m\leq\frac{N_k}{10^{d_{k-1}}}\\m\in\mathbb{N}}}|\cos(10^{d_{k-1}}m\pi\alpha)|.
\end{equation*}
Thus
\begin{align*}
r'&\leq\liminf_{k\rightarrow\infty}\min_{\substack{1\leq m\leq\frac{N_k}{10^{d_{k-1}}}\\m\in\mathbb{N}}}|\cos(10^{d_{k-1}}m\pi\alpha)|^{N_k}\\
&\leq\liminf_{k\rightarrow\infty}\min_{\substack{1\leq m\leq\frac{N_k}{10^{d_{k-1}}}\\m\in\mathbb{N}}}|\cos(10^{d_{k-1}}m\pi\alpha)|^{10^{d_{k-1}}m}.
\end{align*}
Since $r<r'$, it therefore follows that for sufficiently large $k\in\mathbb{N}$ we'll have
\begin{equation*}
r<|\cos(10^{d_{k-1}}m\pi\alpha)|^{10^{d_{k-1}}m}
\end{equation*}
for all natural numbers $m\leq\frac{N_k}{10^{d_{k-1}}}$. Therefore for $k$ sufficiently large we have
\begin{equation*}
|\{n\leq N_k:r<|\cos(n\pi\alpha)|^n\}|>\frac{N_k}{10^{d_{k-1}}}\geq N_kf(N_k).
\end{equation*}
\end{proof}
To prove Theorem \ref{thm5} we again require the following lemma on the number of rational numbers that are relatively close to any given irrational number.
\begin{lemma}\label{lem7}
Let $a>0$ and $\alpha\in\mathbb{R}\setminus\mathbb{Q}$. Then there exists infinitely many $N\in\mathbb{N}$ such that
\begin{equation*}
\left|\left\{n\leq N:\exists m\in\mathbb{Z}\left|\alpha-\frac{m}{n}\right|<\frac{a}{n^{3/2}}\right\}\right|\geq\frac{5^{1/4}a^{1/2}N^{1/4}}{2}.
\end{equation*}
\end{lemma}
\begin{proof}
Let $\alpha\in\mathbb{R}\setminus\mathbb{Q}$. Then there exists infinitely many $u\in\mathbb{Z}$ and $v\in\mathbb{N}$ such that
\begin{equation*}
\left|\alpha-\frac{u}{v}\right|<\frac{1}{\sqrt{5}v^2}.
\end{equation*}
Label such values of $v$ as an increasing sequence $(v_t)_t$ and the corresponding values of $u$ as $(u_t)_t$. Let $t$ be sufficiently large such that $5^{1/3}a^{2/3}v_t^{1/3}>2$. Let
\begin{equation*}
N_t=\lfloor 5^{1/3}a^{2/3}v_t^{1/3}\rfloor v_t.
\end{equation*}
Let $d\in\mathbb{N}$ with
\begin{equation*}
d\leq\lfloor 5^{1/3}a^{2/3}v_t^{1/3}\rfloor.
\end{equation*}
Then we have
\begin{equation*}
\left(dv_t\right)^{3/2}<\sqrt{5}av_t^2<a\left|\alpha-\frac{u_t}{v_t}\right|^{-1}
\end{equation*}
so that
\begin{equation*}
\left|\alpha-\frac{u}{v}\right|<\frac{a}{\left(dv_t\right)^{3/2}}.
\end{equation*}
Letting $d$ run from $1$ to $\lfloor 5^{1/3}a^{2/3}v_t^{1/3}\rfloor$ gives us
\begin{align*}
\left|\left\{n\leq N_t:\exists m\in\mathbb{Z}\left|\alpha-\frac{m}{n}\right|<\frac{a}{n^{3/2}}\right\}\right|&\geq\lfloor 5^{1/3}a^{2/3}v_t^{1/3}\rfloor\\
&\geq\frac{5^{1/3}a^{2/3}v_t^{1/3}}{2}\\
&>\frac{5^{1/4}a^{1/2}N_t^{1/4}}{2}.
\end{align*}
\end{proof}
\begin{proof}[Proof of Theorem \ref{thm5}]
Let $a=\frac{\sqrt{1-r}}{\pi}$ and suppose we have the inequality
\begin{equation*}
\left|\alpha-\frac{m}{n}\right|<\frac{a}{n^{3/2}}
\end{equation*}
for some $m\in\mathbb{Z}$ and $n\in\mathbb{N}$. Then we have the following:
\begin{align*}
|\cos(\pi\alpha n)|^n&=\left|\cos(\pi\alpha n-\pi m)\right|^n\\
&>\cos^n\left(a\pi n^{-1/2}\right)\\
&>\left(1-\frac{\pi^2a^2}{n}\right)^n\\
&>1-\pi^2a^2\\
&=r.
\end{align*}
Thus for all $N\in\mathbb{N}$, we have
\begin{equation*}
|\{n\leq N:r<|\cos(n\pi\alpha)|^n\}|\geq\left|\left\{n\leq N:\exists m\in\mathbb{Z}\left|\alpha-\frac{m}{n}\right|<\frac{a}{n^{3/2}}\right\}\right|.
\end{equation*}
The result follows from Lemma \ref{lem7}.
\end{proof}
\section{Future Work}
Our results give rise to more questions worth exploring. For instance can the lower bound in Theorem \ref{thm5} be improved? How much can it be improved for certain values of $\alpha$ such as algebraic numbers or Liouville numbers? What is an optimal bound that will work for all $n\in\mathbb{N}$? In Theorem \ref{thm2}, can we replace $\cos\left(\alpha n^2\right)$ with $\cos\left(\alpha p(n)\right)$, where $p(n)$ is any given polynomial? 
\subsection*{Acknowledgements}
The author would like to thank the Azrieli Foundation for the award of an Azrieli Fellowship, which made this research possible, as well as Dr. Daniel Berend for helpful comments on this paper.

\end{document}